\documentclass[12pt,twoside,reqno]{amsart}
\usepackage[colorlinks=false,citecolor=blue]{hyperref}
\usepackage{mathptmx, amsmath, amssymb, amsfonts, amsthm, mathptmx, enumerate, color,mathrsfs}
\usepackage{graphicx}
\usepackage{epstopdf}

\newcommand{\vertiii}[1]{{\left\vert\kern-0.25ex\left\vert\kern-0.25ex\left\vert #1 
    \right\vert\kern-0.25ex\right\vert\kern-0.25ex\right\vert}}
\newtheorem{theorem}{Theorem}[section]
\newtheorem{lemma}{Lemma}[section]
\newtheorem{remark}{Remark}[section]

\newtheorem{corollary}{Corollary}[section]

\numberwithin{equation}{section}

\begin{document}
\title{More inner product bounds with applications}
\author{Mohammad Sababheh and Hamid Reza Moradi}

\begin{abstract}
The main goal of this paper is to present new bounds for certain inner products in Hilbert spaces, with applications to the numerical radius and the operator norm. The obtained results significantly improve earlier results in this direction.
\end{abstract}

\maketitle

%------------------------------------------------------------------------------------%
\pagestyle{myheadings} \markboth{\centerline{}}
{\centerline{}} \bigskip \bigskip 
%------------------------------------------------------------------------------------%

\section{Introduction}
Let $\mathbb{H}$ be a complex Hilbert spaces, with inner product $\left<\cdot,\cdot\right>$, and induced norm $\|\cdot\|$. The set of all bounded linear operators from one Hilbert space $\mathbb{H}_1$ to another $\mathbb{H}_2$ will be denoted by $\mathbb{B}(\mathbb{H}_1,\mathbb{H}_2).$ When $\mathbb{H}_1=\mathbb{H}_2=\mathbb{H}$, we simply write $\mathbb{B}(\mathbb{H})$ instead of $\mathbb{B}(\mathbb{H},\mathbb{H})$.

Among the most basic inequalities in Hilbert spaces, and in inner product spaces in general, is the Cauchy-Schwarz inequality that asserts 
\begin{equation}\label{Eq_CS}
|\left<x,y\right>|\leq \|x\|\;\|y\|; x,y\in\mathbb{H}.
\end{equation}
This inequality has been a fundamental foundation of Hilbert space theory, with numerous applications involving it or its variants.

Among the most useful and usable variant of \eqref{Eq_CS} is Buzano inequality \cite{bu}, which  states
\begin{equation}\label{Eq_Buz}
\left| \left\langle x,z \right\rangle  \right|\left| \left\langle y,z \right\rangle  \right|\le \frac{{{\left\| z \right\|}^{2}}}{2}\left( \left| \left\langle x,y \right\rangle  \right|+\left\| x \right\|\left\| y \right\| \right),
\end{equation}
for any $x,y,z\in \mathbb{H}$.

A more elaborated version that extends \eqref{Eq_CS} is the  mixed Cauchy-Schwarz inequality, which states that if $T\in\mathbb{B}(\mathbb{H})$, and $x,y\in\mathbb{H}$, then \cite{kato}
\begin{equation}\label{eq_mixed_t}
{{\left| \left\langle Tx,y \right\rangle  \right|}^{2}}\le \left\langle {{\left| T \right|}^{2\left( 1-t \right)}}x,x \right\rangle \left\langle {{\left| {{T}^{*}} \right|}^{2t}}y,y \right\rangle; 0\leq t\leq 1.
\end{equation}
The fact that \eqref{eq_mixed_t} extends \eqref{Eq_CS} follows on taking $T=I$; the identity operator. In this context, $|T|$ refers to the absolute value operator, defined by $|T|=(T^*T)^{\frac{1}{2}},$ where $T^*$ stands for the adjoint of $T$.

It has been a common practice to associate elements of $\mathbb{B}(\mathbb{H}_1,\mathbb{H}_2)$ with certain scalars, known as norms, for the purpose of comparing operators, since there is no natural total order on the space of bounded linear operators. Given $T\in\mathbb{B}(\mathbb{H}_1,\mathbb{H}_2),$ the operator norm of $T$  is defined by $\|T\|=\sup\{\|Tx\|:x\in\mathbb{H}_1,\|x\|=1\}.$ Another interesting norm is the so-called numerical radius, which is defined on $\mathbb{B}(\mathbb{H})$, by
\[\omega(T)=\sup\{|\left<Tx,x\right>|:x\in\mathbb{H},\|x\|=1\}; T\in\mathbb{B}(\mathbb{H}).\]
It is well known that if $T\in\mathbb{B}(\mathbb{H}),$ one has the equivalence \cite[Theorem 1.3-1]{Gustafson_Book_1997}
\begin{equation}\label{Eq_Equiv}
\frac{1}{2}\|T\|\leq\omega(T)\leq \|T\|.
\end{equation}
One significance of this relation is the way it provides an interval, in terms of $\|T\|$, that contains $\omega(T)$; a quantity that is usually not easy to find, compared with $\|T\|.$ 

Numerous researchers have invested considerable effort to sharpen the bounds in \eqref{Eq_CS}, yielding various forms. The Cauchy-Schwarz inequality and its variants, as above, are unavoidable tools in this investigation.

Of particular interest, matrices, which are operators on a finite-dimensional Hilbert space, have received considerable attention. More particular, matrices with non-negative entries found an easy path because of the observation, proved in \cite{Goldberg_LAA_1975}, that if $T=[t_{ij}]_{i,j=1}^{n}$ is an $n\times n$ matrix of non-negative entries, then $\omega(T)=\|\Re(T)\|,$ where $\Re(T)=\frac{T+T^*}{2}$ is the real part of $T$. 

As one significant application of this non-negativity issue, we mention one application on operator matrices. Let $\mathbb{H}_i$ be Hilbert spaces for $i=1,\ldots,n$, and let $T_{ij}\in\mathbb{B}(\mathbb{H}_j,\mathbb{H}_i)$. The operator matrix $[T_{ij}]$ is then an operator in $\mathbb{B}(\mathbb{H}_1\oplus\mathbb{H}_2\oplus\ldots\oplus\mathbb{H}_n).$ This is indeed a non-easy operator to deal with. However, in \cite{Hou_IEOT_1995}, the following bound was given
\begin{equation}\label{Eq_Hou}
\|[T_{ij}]\|\leq\|[\|T_{ij}\|]\|.
\end{equation}
The difference between the left and right sides of this inequality is that the left side is the norm of an operator matrix, while the right side is the norm of an $n\times n$ matrix of non-negative entries. It has been an interesting topic to discuss possible bounds for $\omega\left([T_{ij}]\right)$ in a way similar to that in \eqref{Eq_Hou}. It is so ambitious to have $\omega\left([T_{ij}]\right)\leq \omega\left([\omega(T_{ij})]\right).$ But this is impossible; as one can verify with the example
\[T=\left[
\begin{array}{cccc}
 -3 & 2 & -1 & -1 \\
 -2 & 2 & 3 & -1 \\
 -2 & 3 & 3 & -2 \\
 1 & 1 & 0 & -2 \\
\end{array}
\right],\]
as a $2\times 2$ operator matrix of square matrices. This challenge urges researchers to find upper bounds for $\omega\left([T_{ij}]\right)$ that simplify its computation. Among the most interesting upper bounds for the numerical radius of an operator matrix, the following was shown in  \cite{Abu-Omar_LAA_2015}, for an $n\times n$ operator matrix $T=[T_{ij}]\in\mathbb{B}(\oplus_{k=1}^{n}\mathbb{H})$:
\begin{equation}  \label{eq_Omar_1}
\omega(T)\leq \omega\left([t_{ij}]\right),\;{\text{where}}\; t_{ij}=\left\{%
\begin{array}{cc}
\omega(T_{ij}), & i=j \\ 
\|T_{ij}\|, & i\not=j%
\end{array}%
\right..
\end{equation}
Some other bounds for the numerical radius of an operator matrix can be found in \cite{Audeh_Filomat_2025, Bhunia_Miskolc_2023, Hajmohamadi_JMI_2018, Hirzallah_IEOT_2011,
Jana_MathSlov_2024, Sababheh_MIA_2024}. As an application of our results, we will be able to find a refined form of \eqref{eq_Omar_1}.

Our primary goal in this paper is to present some bounds for the numerical radius of operator matrices, with applications that involve the spectral radius $r(\cdot)$, which is defined for $T\in\mathbb{B}(\mathbb{H})$ by 
\[r(T)=\sup\{|\lambda|:\lambda\;is\;in\;the\;spectrum\;of\;T\}.\]
To support our results, we will need to prove refined versions of recent inner product inequalities involving both the operator norm and the numerical radius. 

The organization of the subsequent sections will be as follows. In the next section, we prove two bounds for the inner products of the form $\left| \left\langle Ax,y \right\rangle  \right|+\left| \left\langle By,x \right\rangle  \right|$. Then we use these bounds to prove numerical radius and spectral radius bounds, and compare them with existing results. Our results will be compared with results from \cite{Abu-Omar_LAA_2015,2,Hirzallah_IEOT_2011,Hou_IEOT_1995,Kittaneh_Studia_2003,Shebrawi_LAA_2017}.
\section{Two inner product bounds}
In this section we prove two upper bounds for $\left| \left\langle Ax,y \right\rangle  \right|+\left| \left\langle By,x \right\rangle  \right|.$ To show that both bounds are non-trivial, we give a concrete comparison with a result resul, and we show that these two bounds are non-comparable, in general.
\begin{lemma}\label{2}
Let $A\in \mathbb B\left( {{\mathbb H}_{1}},{{\mathbb H}_{2}} \right)$ and $B\in \mathbb B\left( {{\mathbb H}_{2}},{{\mathbb H}_{1}} \right)$. For any $x\in {{\mathbb H}_{1}} $, $y\in {{\mathbb H}_{2}}$, 
\[\left| \left\langle Ax,y \right\rangle  \right|+\left| \left\langle By,x \right\rangle  \right|\le \sqrt{\omega \left( \left| A \right|+i\left| {{B}^{*}} \right| \right)\omega \left( \left| {{A}^{*}} \right|+i\left| B \right| \right)+\left\| A \right\|\left\| B \right\|+\omega \left( BA \right)}\left\| x \right\|\left\| y \right\|.\]
\end{lemma}
\begin{proof}
Indeed, we know that by the Buzano's inequality \eqref{Eq_Buz}, the mixed Cauchy-Schwarz inequality \eqref{eq_mixed_t}, and the Cauchy-Schwarz inequality that
\begin{align}
  & {{\left( \left| \left\langle Ax,y \right\rangle  \right|+\left| \left\langle By,x \right\rangle  \right| \right)}^{2}} \nonumber\\ 
 & ={{\left| \left\langle Ax,y \right\rangle  \right|}^{2}}+{{\left| \left\langle By,x \right\rangle  \right|}^{2}}+2\left| \left\langle Ax,y \right\rangle  \right|\left| \left\langle By,x \right\rangle  \right| \nonumber\\ 
 & ={{\left| \left\langle Ax,y \right\rangle  \right|}^{2}}+{{\left| \left\langle By,x \right\rangle  \right|}^{2}}+2\left| \left\langle Ax,y \right\rangle  \right|\left| \left\langle y,{{B}^{*}}x \right\rangle  \right| \nonumber\\ 
 & ={{\left| \left\langle Ax,y \right\rangle  \right|}^{2}}+{{\left| \left\langle By,x \right\rangle  \right|}^{2}}+2\left| \left\langle Ax,y \right\rangle \left\langle y,{{B}^{*}}x \right\rangle  \right| \nonumber\\ 
 & \le {{\left| \left\langle Ax,y \right\rangle  \right|}^{2}}+{{\left| \left\langle By,x \right\rangle  \right|}^{2}}+\left( \left\| Ax \right\|\left\| {{B}^{*}}x \right\|+\left| \left\langle Ax,{{B}^{*}}x \right\rangle  \right| \right){{\left\| y \right\|}^{2}} \nonumber\\ 
 & ={{\left| \left\langle Ax,y \right\rangle  \right|}^{2}}+{{\left| \left\langle By,x \right\rangle  \right|}^{2}}+\left( \left\| Ax \right\|\left\| {{B}^{*}}x \right\|+\left| \left\langle BAx,x \right\rangle  \right| \right){{\left\| y \right\|}^{2}} \nonumber\\ 
 & \le \left\langle \left| A \right|x,x \right\rangle \left\langle \left| {{A}^{*}} \right|y,y \right\rangle +\left\langle \left| B \right|y,y \right\rangle \left\langle \left| {{B}^{*}} \right|x,x \right\rangle +\left( \left\| Ax \right\|\left\| {{B}^{*}}x \right\|+\left| \left\langle BAx,x \right\rangle  \right| \right){{\left\| y \right\|}^{2}} \nonumber\\ 
 & \le \sqrt{{{\left\langle \left| A \right|x,x \right\rangle }^{2}}+{{\left\langle \left| {{B}^{*}} \right|x,x \right\rangle }^{2}}}\sqrt{{{\left\langle \left| {{A}^{*}} \right|y,y \right\rangle }^{2}}+{{\left\langle \left| B \right|y,y \right\rangle }^{2}}}+\left( \left\| Ax \right\|\left\| {{B}^{*}}x \right\|+\left| \left\langle BAx,x \right\rangle  \right| \right){{\left\| y \right\|}^{2}} \nonumber\\ 
 & =\left| \left\langle \left| A \right|x,x \right\rangle +i\left\langle \left| {{B}^{*}} \right|x,x \right\rangle  \right|\left| \left\langle \left| {{A}^{*}} \right|y,y \right\rangle +i\left\langle \left| B \right|y,y \right\rangle  \right|+\left( \left\| Ax \right\|\left\| {{B}^{*}}x \right\|+\left| \left\langle BAx,x \right\rangle  \right| \right){{\left\| y \right\|}^{2}} \nonumber\\ 
 & =\left| \left\langle \left( \left| A \right|+i\left| {{B}^{*}} \right| \right)x,x \right\rangle  \right|\left| \left\langle \left( \left| {{A}^{*}} \right|+i\left| B \right| \right)y,y \right\rangle  \right|+\left( \left\| Ax \right\|\left\| {{B}^{*}}x \right\|+\left| \left\langle BAx,x \right\rangle  \right| \right){{\left\| y \right\|}^{2}} \label{3}\\ 
 & \le \omega \left( \left| A \right|+i\left| {{B}^{*}} \right| \right)\omega \left( \left| {{A}^{*}} \right|+i\left| B \right| \right){{\left\| x \right\|}^{2}}{{\left\| y \right\|}^{2}}+\left( \left\| A \right\|\left\| {{B}^{*}} \right\|+\omega \left( BA \right) \right){{\left\| x \right\|}^{2}}{{\left\| y \right\|}^{2}} \nonumber\\ 
 & =\left( \omega \left( \left| A \right|+i\left| {{B}^{*}} \right| \right)\omega \left( \left| {{A}^{*}} \right|+i\left| B \right| \right)+\left\| A \right\|\left\| B \right\|+\omega \left( BA \right) \right){{\left\| x \right\|}^{2}}{{\left\| y \right\|}^{2}}\nonumber,
\end{align}
as required.
\end{proof}

\begin{remark}\label{Rem_1_Bet1}
It has been shown in \cite[Proposition 1.4]{1} that if $A,B\in \mathbb B\left( \mathbb H \right)$ are self-adjoint, then
\begin{equation}\label{Eq_Fil}
\omega \left( A+iB \right)\le {{\left\| {{A}^{2}}+{{B}^{2}} \right\|}^{\frac{1}{2}}}.
\end{equation}
From Lemma \ref{2}, we infer that
\[\begin{aligned}
  & \left| \left\langle Ax,y \right\rangle  \right|+\left| \left\langle By,x \right\rangle  \right| \\ 
 & \le \sqrt{\omega \left( \left| A \right|+i\left| {{B}^{*}} \right| \right)\omega \left( \left| {{A}^{*}} \right|+i\left| B \right| \right)+\left\| A \right\|\left\| B \right\|+\omega \left( BA \right)}\left\| x \right\|\left\| y \right\| \\ 
 & \le \sqrt{\sqrt{\left\| {{\left| A \right|}^{2}}+{{\left| {{B}^{*}} \right|}^{2}} \right\|\left\| {{\left| {{A}^{*}} \right|}^{2}}+{{\left| B \right|}^{2}} \right\|}+\left\| A \right\|\left\| B \right\|+\omega \left( BA \right)}\left\| x \right\|\left\| y \right\| \\ 
 & \le \sqrt{\sqrt{\left( \left\| {{\left| A \right|}^{2}} \right\|+\left\| {{\left| {{B}^{*}} \right|}^{2}} \right\| \right)\left( \left\| {{\left| {{A}^{*}} \right|}^{2}} \right\|+\left\| {{\left| B \right|}^{2}} \right\| \right)}+\left\| A \right\|\left\| B \right\|+\omega \left( BA \right)}\left\| x \right\|\left\| y \right\| \\ 
 & =\sqrt{\sqrt{\left( {{\left\| \;\left| A \right| \;\right\|}^{2}}+{{\left\| \;\left| {{B}^{*}} \right| \;\right\|}^{2}} \right)\left( {{\left\| \;\left| {{A}^{*}} \right| \;\right\|}^{2}}+{{\left\|\; \left| B \right| \;\right\|}^{2}} \right)}+\left\| A \right\|\left\| B \right\|+\omega \left( BA \right)}\left\| x \right\|\left\| y \right\| \\ 
 & =\sqrt{{{\left\| A \right\|}^{2}}+{{\left\| B \right\|}^{2}}+\left\| A \right\|\left\| B \right\|+\omega \left( BA \right)}\left\| x \right\|\left\| y \right\| \\ 
 & =\sqrt{{{\left( \left\| A \right\|+\left\| B \right\| \right)}^{2}}-\left( \left\| A \right\|\left\| B \right\|-\omega \left( BA \right) \right)}\left\| x \right\|\left\| y \right\|.
\end{aligned}\]
Indeed, our result nicely improves
\[\left| \left\langle Ax,y \right\rangle  \right|+\left| \left\langle By,x \right\rangle  \right|\le \sqrt{{{\left( \left\| A \right\|+\left\| B \right\| \right)}^{2}}-\left( \left\| A \right\|\left\| B \right\|-\omega \left( BA \right) \right)}\left\| x \right\|\left\| y \right\|,\]
which was proved in \cite[Lemma 2.2]{2}.
\end{remark}

Our second inner product bound can be stated as follows.

\begin{lemma}\label{4}
Let $A\in \mathbb B\left( {{\mathbb H}_{1}},{{\mathbb H}_{2}} \right)$ and $B\in \mathbb B\left( {{\mathbb H}_{2}},{{\mathbb H}_{1}} \right)$. For any $x\in {{\mathbb H}_{1}} $, $y\in {{\mathbb H}_{2}}$, 
\[\left| \left\langle Ax,y \right\rangle  \right|+\left| \left\langle By,x \right\rangle  \right|\le \sqrt{\omega \left( \left| A \right|+i\left| {{B}^{*}} \right| \right)\omega \left( \left| {{A}^{*}} \right|+i\left| B \right| \right)+\frac{1}{2}\left\| {{\left| A \right|}^{2}}+{{\left| {{B}^{*}} \right|}^{2}} \right\|+\omega \left( BA \right)}\left\| x \right\|\left\| y \right\|.\]
\end{lemma}
\begin{proof}
It follows from \eqref{3} that
\[\begin{aligned}
  & {{\left( \left| \left\langle Ax,y \right\rangle  \right|+\left| \left\langle By,x \right\rangle  \right| \right)}^{2}} \\ 
 & \le \left| \left\langle \left( \left| A \right|+i\left| {{B}^{*}} \right| \right)x,x \right\rangle  \right|\left| \left\langle \left( \left| {{A}^{*}} \right|+i\left| B \right| \right)y,y \right\rangle  \right|+\left( \left\| Ax \right\|\left\| {{B}^{*}}x \right\|+\left| \left\langle BAx,x \right\rangle  \right| \right){{\left\| y \right\|}^{2}} \\ 
 & =\left| \left\langle \left( \left| A \right|+i\left| {{B}^{*}} \right| \right)x,x \right\rangle  \right|\left| \left\langle \left( \left| {{A}^{*}} \right|+i\left| B \right| \right)y,y \right\rangle  \right|+\left( \sqrt{\left\langle Ax,Ax \right\rangle \left\langle {{B}^{*}}x,{{B}^{*}}x \right\rangle }+\left| \left\langle BAx,x \right\rangle  \right| \right){{\left\| y \right\|}^{2}} \\ 
 & =\left| \left\langle \left( \left| A \right|+i\left| {{B}^{*}} \right| \right)x,x \right\rangle  \right|\left| \left\langle \left( \left| {{A}^{*}} \right|+i\left| B \right| \right)y,y \right\rangle  \right|+\left( \sqrt{\left\langle {{\left| A \right|}^{2}}x,x \right\rangle \left\langle {{\left| {{B}^{*}} \right|}^{2}}x,x \right\rangle }+\left| \left\langle BAx,x \right\rangle  \right| \right){{\left\| y \right\|}^{2}} \\ 
 & \le \left| \left\langle \left( \left| A \right|+i\left| {{B}^{*}} \right| \right)x,x \right\rangle  \right|\left| \left\langle \left( \left| {{A}^{*}} \right|+i\left| B \right| \right)y,y \right\rangle  \right|+\left( \frac{1}{2}\left( \left\langle {{\left| A \right|}^{2}}x,x \right\rangle +\left\langle {{\left| {{B}^{*}} \right|}^{2}}x,x \right\rangle  \right)+\left| \left\langle BAx,x \right\rangle  \right| \right){{\left\| y \right\|}^{2}} \\ 
 & =\left| \left\langle \left( \left| A \right|+i\left| {{B}^{*}} \right| \right)x,x \right\rangle  \right|\left| \left\langle \left( \left| {{A}^{*}} \right|+i\left| B \right| \right)y,y \right\rangle  \right|+\left( \frac{1}{2}\left\langle \left( {{\left| A \right|}^{2}}+{{\left| {{B}^{*}} \right|}^{2}} \right)x,x \right\rangle +\left| \left\langle BAx,x \right\rangle  \right| \right){{\left\| y \right\|}^{2}} \\ 
 & \le \omega \left( \left| A \right|+i\left| {{B}^{*}} \right| \right)\omega \left( \left| {{A}^{*}} \right|+i\left| B \right| \right){{\left\| x \right\|}^{2}}{{\left\| y \right\|}^{2}}+\left( \frac{1}{2}\left\| {{\left| A \right|}^{2}}+{{\left| {{B}^{*}} \right|}^{2}} \right\|+\omega \left( BA \right) \right){{\left\| x \right\|}^{2}}{{\left\| y \right\|}^{2}} \\ 
 & =\left( \omega \left( \left| A \right|+i\left| {{B}^{*}} \right| \right)\omega \left( \left| {{A}^{*}} \right|+i\left| B \right| \right)+\frac{1}{2}\left\| {{\left| A \right|}^{2}}+{{\left| {{B}^{*}} \right|}^{2}} \right\|+\omega \left( BA \right) \right){{\left\| x \right\|}^{2}}{{\left\| y \right\|}^{2}}, 
\end{aligned}\]
as required.
\end{proof}

\begin{remark}\label{Rem_222}
In both Lemmas \ref{2} and {4}, we found upper bounds for $\left| \left\langle Ax,y \right\rangle  \right|+\left| \left\langle By,x \right\rangle  \right|$. In this remark, we give two examples to show that neither bound can always be better than the other. First, let
\[A=\left[
\begin{array}{cc}
 4 & 1 \\
 3 & 3 \\
\end{array}
\right], B=\left[
\begin{array}{cc}
 4& 1 \\
 -3 & -1 \\
\end{array}
\right], x=\left[
\begin{array}{c}
 -\frac{1}{\sqrt{2}} \\
 \frac{1}{\sqrt{2}} \\
\end{array}
\right], y=\left[
\begin{array}{c}
 -\frac{4}{\sqrt{17}} \\
 -\frac{1}{\sqrt{17}} \\
\end{array}
\right].\]
It can be seen that
\[\sqrt{\omega \left( \left| A \right|+i\left| {{B}^{*}} \right| \right)\omega \left( \left| {{A}^{*}} \right|+i\left| B \right| \right)+\left\| A \right\|\left\| B \right\|+\omega \left( BA \right)}\left\| x \right\|\left\| y \right\|\approx 9.567,\]
while
\[\sqrt{\omega \left( \left| A \right|+i\left| {{B}^{*}} \right| \right)\omega \left( \left| {{A}^{*}} \right|+i\left| B \right| \right)+\frac{1}{2}\left\| {{\left| A \right|}^{2}}+{{\left| {{B}^{*}} \right|}^{2}} \right\|+\omega \left( BA \right)}\left\| x \right\|\left\| y \right\|\approx 9.10612,\]
showing that Lemma \ref{4} is better in this case.\\
On the other hand, letting
\[A=\left[
\begin{array}{cc}
 2 & 2 \\
 -1 & 2 \\
\end{array}
\right], B=\left[
\begin{array}{cc}
 3 & 4 \\
 4 & 1 \\
\end{array}
\right], x=\left[
\begin{array}{c}
 -\frac{3}{\sqrt{10}} \\
 -\frac{1}{\sqrt{10}} \\
\end{array}
\right], y=\left[
\begin{array}{c}
 \frac{3}{\sqrt{13}} \\
 -\frac{2}{\sqrt{13}} \\
\end{array}
\right],\]
we find that
\[\sqrt{\omega \left( \left| A \right|+i\left| {{B}^{*}} \right| \right)\omega \left( \left| {{A}^{*}} \right|+i\left| B \right| \right)+\left\| A \right\|\left\| B \right\|+\omega \left( BA \right)}\left\| x \right\|\left\| y \right\|\approx 9.02776,\]
while
\[\sqrt{\omega \left( \left| A \right|+i\left| {{B}^{*}} \right| \right)\omega \left( \left| {{A}^{*}} \right|+i\left| B \right| \right)+\frac{1}{2}\left\| {{\left| A \right|}^{2}}+{{\left| {{B}^{*}} \right|}^{2}} \right\|+\omega \left( BA \right)}\left\| x \right\|\left\| y \right\|\approx 9.27186,\]
showing that Lemma \ref{2} is better in this case.

Of course, this comparison is based on the comparison between the two quantities $\|A\|\;\|B\|$ and 
$\frac{1}{2}\left\| {{\left| A \right|}^{2}}+{{\left| {{B}^{*}} \right|}^{2}} \right\|$. These two quantities are, in general, incomparable, as one can verify with the above choices.
\end{remark}

\section{Applications towards the numerical radius and spectral radius}

In this section, we present our main findings, including some bounds for operator matrices, which are compared with previously known bounds. Interestingly, applying Lemma \ref{2} implies the following bound, while Lemma \ref{4} implies another bound, as we will see below. Simpler cases will be treated in separate lemmas for better exposition.

\begin{theorem}\label{5}
Let $\left[ {{T}_{ij}} \right]$ be an $n\times n$ operator matrix with ${{T}_{ij}}\in \mathbb B\left( {{\mathbb H}_{j}},{{\mathbb H}_{i}} \right)$. Then
\[\omega \left( \left[ {{T}_{ij}} \right] \right)\le \omega \left( \left[ {{\alpha }_{ij}} \right]_{i,j=1}^{n} \right),\]
where
\begin{equation*}
{{\alpha }_{ij}}=\left\{ \begin{array}{ll}
   \omega \left( {{T}_{ii}} \right)&\text{ if }i=j \\ 
  \sqrt{\omega \left( \left| {{T}_{ij}} \right|+i\left| T_{ji}^{*} \right| \right)\omega \left( \left| T_{ij}^{*} \right|+i\left| {{T}_{ji}} \right| \right)+\left\| {{T}_{ij}} \right\|\left\| {{T}_{ji}} \right\|+\omega \left( {{T}_{ji}}{{T}_{ij}} \right)}&\text{ if }i<j \\ 
  0&\text{ if }i>j \\ 
\end{array} \right..
\end{equation*}
\end{theorem}
\begin{proof}
Let $x=\oplus_{i=1}^{n}x_i\in\oplus_{i=1}^{n}\mathbb{H}_i$ be a unit vector, so that $\sum_{i=1}^{n}\|x_i\|^2=1.$ We have by Lemma \ref{2} that
\[\begin{aligned}
  & \left| \sum\limits_{i,j=1}^{n}{\left\langle {{T}_{ij}}{{x}_{j}},{{x}_{i}} \right\rangle } \right| \\ 
 & \le \sum\limits_{i,j=1}^{n}{\left| \left\langle {{T}_{ij}}{{x}_{j}},{{x}_{i}} \right\rangle  \right|} \\ 
 & = \sum\limits_{j=1}^{n}{\left| \left\langle {{T}_{jj}}{{x}_{j}},{{x}_{j}} \right\rangle  \right|}+\sum\limits_{\underset{i\ne j}{\mathop{i,j=1}}\,}^{n}{\left| \left\langle {{T}_{ij}}{{x}_{j}},{{x}_{i}} \right\rangle  \right|} \\ 
 & =\sum\limits_{j=1}^{n}{\left| \left\langle {{T}_{jj}}{{x}_{j}},{{x}_{j}} \right\rangle  \right|}+\sum\limits_{\underset{i<j}{\mathop{i,j=1}}\,}^{n}\left({\left| \left\langle {{T}_{ij}}{{x}_{j}},{{x}_{i}} \right\rangle  \right|+\left| \left\langle {{T}_{ji}}{{x}_{i}},{{x}_{j}} \right\rangle  \right|}\right) \\ 
 & \le \sum\limits_{j=1}^{n}{\omega \left( {{T}_{jj}} \right){{\left\| {{x}_{j}} \right\|}^{2}}}+\sum\limits_{\underset{i<j}{\mathop{i,j=1}}\,}^{n}{{{\alpha }_{ij}}\left\| {{x}_{i}} \right\|\left\| {{x}_{j}} \right\|} \\ 
 & =\left\langle \left[ {{\alpha }_{ij}} \right]_{i,j=1}^{n}\left| x \right|,\left| x \right| \right\rangle,
\end{aligned}\]
where $\left| x \right|={{\left( \left\| {{x}_{1}} \right\|,\left\| {{x}_{2}} \right\|,\ldots ,\left\| {{x}_{n}} \right\| \right)}^{T}}\in {{\mathbb{C}}^{n}}$ is a unit vector, and $\alpha_{ij}$ is as given in the statement of the theorem. From this, we obtain
\[\left| \sum\limits_{i,j=1}^{n}{\left\langle {{T}_{ij}}{{x}_{j}},{{x}_{i}} \right\rangle } \right|\le \omega \left( \left[ {{\alpha }_{ij}} \right]_{i,j=1}^{n} \right).\] Hence 
\[\omega \left( \left[ {{T}_{ij}} \right] \right)\le \omega \left( \left[ {{\alpha }_{ij}} \right]_{i,j=1}^{n} \right),\]
as required.
\end{proof}

\begin{remark}
Following \eqref{Eq_Fil} and the calculations in Remark \ref{Rem_1_Bet1}, we see that
\[\sqrt{\omega \left( \left| {{T}_{ij}} \right|+i\left| T_{ji}^{*} \right| \right)\omega \left( \left| T_{ij}^{*} \right|+i\left| {{T}_{ji}} \right| \right)+\left\| {{T}_{ij}} \right\|\left\| {{T}_{ji}} \right\|+\omega \left( {{T}_{ji}}{{T}_{ij}} \right)}\leq \|T_{ij}\|+\|T_{ji}\|.\]
Let $\gamma_{ij}=\left\{\begin{array}{cc}\omega(T_{ii}),&i=j\\\|T_{ij}\|+\|T_{ji}\|,&i<j\\0,&i>j\end{array}\right..$ Then $\alpha_{ij}\leq \gamma_{ij},$ where $\alpha_{ij}$ is as in Theorem \ref{5}. Since both $\alpha_{ij},\gamma_{ij}\geq 0$ and $\alpha_{ij}\leq \gamma_{ij}$, it follows that $\omega([\alpha_{ij}])\leq \omega([\gamma_{ij}]);$ see \cite{Goldberg_LAA_1975} and \cite[Corollary 2.1]{Turman_RIM_2025}. On the other hand, since $\gamma_{ij}\geq 0$, it follows that (see see \cite{Goldberg_LAA_1975,Turman_RIM_2025})
\begin{align*}
\omega([\gamma_{ij}])=\left\|\Re[\gamma_{ij}]\right\|=\left\|\Re[t_{ij}]\right\|=\omega([t_{ij}]),
\end{align*}
where \[ t_{ij}=\left\{%
\begin{array}{cc}
\omega(T_{ij}), & i=j \\ 
\|T_{ij}\|, & i\not=j%
\end{array}%
\right..\]
Thus, we have shown that the bound we found in Theorem \ref{5} is sharper than that in \eqref{eq_Omar_1}.

Furthermore, it was shown in \cite{Hou_IEOT_1995} that $\omega([T_{ij}])\leq\omega([\|T_{ij}\|]).$ It can be easily seen that \eqref{eq_Omar_1} is a refinement of this celebrated bound. Thus, our bound in Theorem \ref{5} is also a refinement of this latter bound.
\end{remark}

Now Lemma \ref{4} implies the following other bound, whose proof is identical to that of Theorem \ref{5}, but implementing Lemma \ref{4} instead of Lemma \ref{2}.
\begin{theorem}\label{6}
Let $\left[ {{T}_{ij}} \right]$ be an $n\times n$ operator matrix with ${{T}_{ij}}\in \mathbb B\left( {{\mathbb H}_{j}},{{\mathbb H}_{i}} \right)$. Then
\[\omega \left( \left[ {{T}_{ij}} \right] \right)\le \omega \left( \left[ {{\beta }_{ij}} \right]_{i,j=1}^{n} \right),\]
where
\begin{equation*}
{{\beta }_{ij}}=\left\{ \begin{array}{ll}
   \omega \left( {{T}_{ii}} \right)&\text{ if }i=j \\ 
  \sqrt{\omega \left( \left| {{T}_{ij}} \right|+i\left| T_{ji}^{*} \right| \right)\omega \left( \left| T_{ij}^{*} \right|+i\left| {{T}_{ji}} \right| \right)+\frac{1}{2}\left\| {{\left| {{T}_{ij}} \right|}^{2}}+{{\left| T_{ji}^{*} \right|}^{2}} \right\|+\omega \left( {{T}_{ji}}{{T}_{ij}} \right)}&\text{ if }i<j \\ 
  0&\text{ if }i>j \\ 
\end{array} \right..
\end{equation*}
\end{theorem}

From Theorem \ref{5}, we obtain the following simple form of the numerical radius of a $2\times 2$ operator matrix.
\begin{corollary}\label{Cor_1}
Let $A\in \mathbb B\left( {{\mathbb H}_{1}} \right)$, $B\in \mathbb B\left( {{\mathbb H}_{2}},{{\mathbb H}_{1}} \right)$, $C\in \mathbb B\left( {{\mathbb H}_{1}},{{\mathbb H}_{2}} \right)$, and $D\in \mathbb B\left( {{\mathbb H}_{2}} \right)$. Then
\[\begin{aligned}
   \omega \left( \left[ \begin{matrix}
   A & B  \\
   C & D  \\
\end{matrix} \right] \right)&\le \frac{1}{2}\left( \omega \left( A \right)+\omega \left( D \right) \right) \\ 
 &\quad +\frac{1}{2}\sqrt{{{\left( \omega \left( A \right)-\omega \left( D \right) \right)}^{2}}+\omega \left( \left| B \right|+i\left| {{C}^{*}} \right| \right)\omega \left( \left| {{B}^{*}} \right|+i\left| C \right| \right)+\left\| B \right\|\left\| C \right\|+\omega \left( CB \right)}.  
\end{aligned}\]
\end{corollary}
\begin{remark}
In \cite{Hirzallah_IEOT_2011}, it was shown that
\begin{equation}\label{Eq_Hirz_op2}
 \omega \left( \left[ \begin{matrix}
   A & B  \\
   C & D  \\
\end{matrix} \right] \right)\le\max\{\omega(A),\omega(D)\}+\frac{\omega(B+C)+\omega(B-C)}{2}.
\end{equation}
If we let 
\[A=\left[
\begin{array}{cc}
 -1 & -2 \\
 -1 & 2 \\
\end{array}
\right], B=\left[
\begin{array}{cc}
 2 & 1 \\
 0 & -2 \\
\end{array}
\right], C=\left[
\begin{array}{cc}
 -3 & -1 \\
 -3 & -3 \\
\end{array}
\right], D=\left[
\begin{array}{cc}
 -2 & 3 \\
 -3 & 0 \\
\end{array}
\right],\]
we find that
\[\max\{\omega(A),\omega(D)\}+\frac{\omega(B+C)+\omega(B-C)}{2}\approx 9.03276,\]
while
\[\frac{1}{2}\left( \omega \left( A \right)+\omega \left( D \right) \right) 
 +\frac{1}{2}\sqrt{{{\left( \omega \left( A \right)-\omega \left( D \right) \right)}^{2}}+\omega \left( \left| B \right|+i\left| {{C}^{*}} \right| \right)\omega \left( \left| {{B}^{*}} \right|+i\left| C \right| \right)+\left\| B \right\|\left\| C \right\|+\omega \left( CB \right)}\approx 6.50583,\]
 showing how the bound in Corollary \ref{Cor_1} can significantly be better than that in \eqref{Eq_Hirz_op2}.  We point out that all numerical examples we attempted indicate that our bound is better than that in \eqref{Eq_Hirz_op2}; however, we still cannot prove this claim rigorously. 
\end{remark}
\color{black}
\begin{remark}
If we let $B=C=O,$ the zero operator, in Corollary \ref{Cor_1}, we deduce that 
\[\omega \left( \left[ \begin{matrix}
   A & O  \\
   O & D  \\
\end{matrix} \right] \right)\le\max\{\omega(A),\omega(D)\}.\]
On the other hand, if we let $A=D=O,$ we infer that

\begin{equation}\label{Eq_Ned_op1}
\begin{aligned}
   \omega \left( \left[ \begin{matrix}
   O & B  \\
   C & O  \\
\end{matrix} \right] \right)&\le \frac{1}{2}\sqrt{\omega \left( \left| B \right|+i\left| {{C}^{*}} \right| \right)\omega \left( \left| {{B}^{*}} \right|+i\left| C \right| \right)+\left\| B \right\|\left\| C \right\|+\omega \left( CB \right)}.  
\end{aligned}
\end{equation}
 Following the same ideas as in Remark \ref{Rem_1_Bet1}, we can show that
\[\begin{aligned}
&\sqrt{\omega \left( \left| B \right|+i\left| {{C}^{*}} \right| \right)\omega \left( \left| {{B}^{*}} \right|+i\left| C \right| \right)+\left\| B \right\|\left\| C \right\|+\omega \left( CB \right)}\\
&\leq \|B\|+\|C\|.
\end{aligned}
\]
This means that we have shown
\[
\begin{aligned}
\omega \left( \left[ \begin{matrix}
   O & B  \\
   C & O  \\
\end{matrix} \right] \right)&\le \frac{1}{2}\sqrt{\omega \left( \left| B \right|+i\left| {{C}^{*}} \right| \right)\omega \left( \left| {{B}^{*}} \right|+i\left| C \right| \right)+\left\| B \right\|\left\| C \right\|+\omega \left( CB \right)}\\
&\leq\frac{\|B\|+\|C\|}{2}.
\end{aligned}
\]
This indeed provides an interesting refinement of the celebrated bound
\[\omega \left( \left[ \begin{matrix}
   O & B  \\
   C & O  \\
\end{matrix} \right] \right)\le\frac{\|B\|+\|C\|}{2},\]
which has been shown as one of the sharpest bounds for $\omega \left( \left[ \begin{matrix}
   O & B  \\
   C & O  \\
\end{matrix} \right] \right)$ in \cite{Hirzallah_IEOT_2011}. 
\end{remark}

\begin{remark}
If we let $C=D=O$ in Corollary \ref{Cor_1}, we reach
\begin{equation}\label{Eq_Ned_Comp_Sh1}
\omega\left(\left[\begin{matrix}A&B\\O&O\end{matrix}\right]\right]\leq \frac{1}{2}\left(\omega(A)+\sqrt{\omega(A)^2+\|B\|^2}\right).
\end{equation}
On the other hand, it has been shown in \cite{Shebrawi_LAA_2017} that
\begin{equation}\label{Eq_Ned_Comp_Sh2}
\omega\left(\left[\begin{matrix}A&B\\O&O\end{matrix}\right]\right]\leq \frac{1}{2}\left(\|A\|+\|AA^*+BB^*\|^{\frac{1}{2}}\right).
\end{equation}
If we let $A=\left[
\begin{array}{cc}
 3 & 2 \\
 -1 & -3 \\
\end{array}
\right]$ and $B=\left[
\begin{array}{cc}
 1 & 1 \\
 1 & -1 \\
\end{array}
\right],$ we find that
\[\frac{1}{2}\left(\omega(A)+\sqrt{\omega(A)^2+\|B\|^2}\right)\approx 3.19774, \frac{1}{2}\left(\|A\|+\|AA^*+BB^*\|^{\frac{1}{2}}\right)\approx 4.64893,\]
which shows that \eqref{Eq_Ned_Comp_Sh1} can provide better estimates than \eqref{Eq_Ned_Comp_Sh2}. We emphasize that this is not always the case, as other examples reveal.

Moreover, it can be seen that for the above $A$ and $B$, the inequality \eqref{Eq_Ned_Comp_Sh1} is indeed an equality, showing that it is a sharp inequality.

\end{remark}

\begin{remark}
In \cite{Shebrawi_LAA_2017}, it has been shown that
\begin{equation}\label{Eq_Shebrawi_3}
\omega\!\left(
\begin{bmatrix}
A & B \\
C & D
\end{bmatrix}
\right)
\leq \frac{1}{2} \omega(A)
+ \frac{1}{2} \omega(D)
+ \frac{1}{4} \bigl\| I + AA^{*} + BB^{*} \bigr\|
+ \frac{1}{4} \bigl\| I + CC^{*} + DD^{*} \bigr\|.
\end{equation}
If we let \[A=\left[
\begin{array}{cc}
 2 & 1 \\
 -1 & -3 \\
\end{array}
\right], B=\left[
\begin{array}{cc}
 -2 & 0 \\
 -3 & 3 \\
\end{array}
\right], C=\left[
\begin{array}{cc}
 2 & 1 \\
 -3 & -3 \\
\end{array}
\right], D=\left[
\begin{array}{cc}
 2 & -3 \\
 3 & -2 \\
\end{array}
\right],\]
we can see that
\[\frac{1}{2} \omega(A)
+ \frac{1}{2} \omega(D)
+ \frac{1}{4} \bigl\| I + AA^{*} + BB^{*} \bigr\|
+ \frac{1}{4} \bigl\| I + CC^{*} + DD^{*} \bigr\|\approx 18.454,\]
while
\[\begin{aligned}
  \frac{1}{2}&\left( \omega \left( A \right)+\omega \left( D \right) \right) \\ 
 &\quad +\frac{1}{2}\sqrt{{{\left( \omega \left( A \right)-\omega \left( D \right) \right)}^{2}}+\omega \left( \left| B \right|+i\left| {{C}^{*}} \right| \right)\omega \left( \left| {{B}^{*}} \right|+i\left| C \right| \right)+\left\| B \right\|\left\| C \right\|+\omega \left( CB \right)}\approx 7.41238.  
\end{aligned}\]
Therefore, Corollary \ref{Cor_1} provides a better estimate than \eqref{Eq_Shebrawi_3} in this example.
\end{remark}

As an immediate consequence of \eqref{Eq_Ned_op1}, we can have the following upper bound for $\omega(B)$. This follows from \eqref{Eq_Ned_op1} on letting $C=B.$ While this provides an upper bound for $\omega(B)$, it also provides a reversed form of the power inequality $\omega(B^2)\leq \omega(B)^2.$
\begin{corollary}\label{7}
Let $B\in \mathbb B\left( {{\mathbb H}} \right)$. Then
\[\omega \left( B \right)\le \frac{1}{2}\sqrt{{{\omega }^{2}}\left( \left| B \right|+i\left| {{B}^{*}} \right| \right)+{{\left\| B \right\|}^{2}}+\omega \left( {{B}^{2}} \right)}.\]
\end{corollary}
\begin{proof}
Letting $C=B$ in \eqref{Eq_Ned_op1}, we get
	\[\begin{aligned}
   \omega \left( B \right)&=\omega \left( \left[ \begin{matrix}
   O & B  \\
   B & O  \\
\end{matrix} \right] \right) \\ 
 & \le \frac{1}{2}\sqrt{\omega \left( \left| B \right|+i\left| {{B}^{*}} \right| \right)\omega \left( \left| {{B}^{*}} \right|+i\left| B \right| \right)+{{\left\| B \right\|}^{2}}+\omega \left( {{B}^{2}} \right)} \\ 
 & =\frac{1}{2}\sqrt{\omega \left( \left| B \right|+i\left| {{B}^{*}} \right| \right)\omega \left( i\left( \left| {{B}^{*}} \right|+i\left| B \right| \right) \right)+{{\left\| B \right\|}^{2}}+\omega \left( {{B}^{2}} \right)} \\ 
 & =\frac{1}{2}\sqrt{\omega \left( \left| B \right|+i\left| {{B}^{*}} \right| \right)\omega \left( i\left| {{B}^{*}} \right|-\left| B \right| \right)+{{\left\| B \right\|}^{2}}+\omega \left( {{B}^{2}} \right)} \\ 
 & =\frac{1}{2}\sqrt{\omega \left( \left| B \right|+i\left| {{B}^{*}} \right| \right)\omega \left( -\left( i\left| {{B}^{*}} \right|-\left| B \right| \right) \right)+{{\left\| B \right\|}^{2}}+\omega \left( {{B}^{2}} \right)} \\ 
 & =\frac{1}{2}\sqrt{\omega \left( \left| B \right|+i\left| {{B}^{*}} \right| \right)\omega \left( -i\left| {{B}^{*}} \right|+\left| B \right| \right)+{{\left\| B \right\|}^{2}}+\omega \left( {{B}^{2}} \right)} \\ 
 & =\frac{1}{2}\sqrt{\omega \left( \left| B \right|+i\left| {{B}^{*}} \right| \right)\omega \left( {{\left( -i\left| {{B}^{*}} \right|+\left| B \right| \right)}^{*}} \right)+{{\left\| B \right\|}^{2}}+\omega \left( {{B}^{2}} \right)} \\ 
 & =\frac{1}{2}\sqrt{\omega \left( \left| B \right|+i\left| {{B}^{*}} \right| \right)\omega \left( i\left| {{B}^{*}} \right|+\left| B \right| \right)+{{\left\| B \right\|}^{2}}+\omega \left( {{B}^{2}} \right)} \\ 
 & =\frac{1}{2}\sqrt{{{\omega }^{2}}\left( \left| B \right|+i\left| {{B}^{*}} \right| \right)+{{\left\| B \right\|}^{2}}+\omega \left( {{B}^{2}} \right)},  
\end{aligned}\]
which completes the proof.
\end{proof}
\begin{remark}
Utilizing \eqref{Eq_Fil}, and following the calculations in Remark \ref{Rem_1_Bet1}, we can see that
\begin{align*}
\sqrt{{{\omega }^{2}}\left( \left| B \right|+i\left| {{B}^{*}} \right| \right)+{{\left\| B \right\|}^{2}}+\omega \left( {{B}^{2}} \right)}&\leq \sqrt{\|\;|B|^2\|+\|\;|B^*|^2\|+\|B\|^2+\omega(B^2)}\\
&\leq \sqrt{\|B\|^2+\|B^*\|^2+\|B\|^2+\|B^2\|}\\
&\leq 2\|B\|.
\end{align*}
Therefore, we deduce from Corollary \ref{7} that
\begin{align*}
\omega \left( B \right)&\le \frac{1}{2}\sqrt{{{\omega }^{2}}\left( \left| B \right|+i\left| {{B}^{*}} \right| \right)+{{\left\| B \right\|}^{2}}+\omega \left( {{B}^{2}} \right)}\\
&\leq\|B\|,
\end{align*}
which is a refinement of the second inequality in \eqref{Eq_Equiv}.

A celebrated sharp upper bound for $\omega(B)$ is given by \cite{Kittaneh_Studia_2003}
\begin{equation}\label{Eq_Kitt_Stud}
\omega(B)\leq\frac{1}{2}\|\;|B|+|B^*|\;\|.
\end{equation}
If we let $B=\left[
\begin{array}{cc}
 -4 & 7 \\
 -4 & -8 \\
\end{array}
\right],$ we find that
\[\omega(B)\approx 8.69626, \frac{1}{2}\sqrt{{{\omega }^{2}}\left( \left| B \right|+i\left| {{B}^{*}} \right| \right)+{{\left\| B \right\|}^{2}}+\omega \left( {{B}^{2}} \right)}\approx 9.74488,\]
and $\frac{1}{2}\|\;|B|+|B^*|\;\|\approx 9.9823,$ showing that the bound we have in Corollary \ref{7} can provide better estimates than \eqref{Eq_Kitt_Stud}. However, this is not always the case, as other examples reveal the opposite conclusion.
\end{remark}

\color{black}
From Theorem \ref{6}, we may state the following $2\times 2$ version. 
\begin{corollary}\label{Cor_2}
Let $A\in \mathbb B\left( {{\mathbb H}_{1}} \right)$, $B\in \mathbb B\left( {{\mathbb H}_{2}},{{\mathbb H}_{1}} \right)$, $C\in \mathbb B\left( {{\mathbb H}_{1}},{{\mathbb H}_{2}} \right)$, and $D\in \mathbb B\left( {{\mathbb H}_{2}} \right)$. Then
{\small
\[\begin{aligned}
   \omega \left( \left[ \begin{matrix}
   A & B  \\
   C & D  \\
\end{matrix} \right] \right)&\le \frac{1}{2}\left( \omega \left( A \right)+\omega \left( D \right) \right) \\ 
 &\quad +\frac{1}{2}\sqrt{{{\left( \omega \left( A \right)-\omega \left( D \right) \right)}^{2}}+\omega \left( \left| B \right|+i\left| {{C}^{*}} \right| \right)\omega \left( \left| {{B}^{*}} \right|+i\left| C \right| \right)+\frac{1}{2}\left\| {{\left| B \right|}^{2}}+{{\left| {{C}^{*}} \right|}^{2}} \right\|+\omega \left( CB \right)}.  
\end{aligned}\]}
\end{corollary}
\begin{remark}
Based on our discussion in Remark \ref{Rem_222}, we can easily see that the two bounds in Corollaries \ref{Cor_1} and \ref{Cor_2} are incomparable, in general.
\end{remark}

Letting $A=D=O$ in Corollary $2\times 2$, and following a similar approach as in Corollary \ref{2}, we can state the following.
\begin{corollary}\label{8}
Let $T\in \mathbb B\left( {{\mathbb H}} \right)$. Then
\[\omega \left( T \right)\le \frac{1}{2}\sqrt{{{\omega }^{2}}\left( \left| T \right|+i\left| {{T}^{*}} \right| \right)+\frac{1}{2}\left\| {{\left| T \right|}^{2}}+{{\left| {{T}^{*}} \right|}^{2}} \right\|+\omega \left( {{T}^{2}} \right)}.\]
\end{corollary}

\begin{remark}
Of course Corollary \ref{8} improves Corollary \ref{7}.
\end{remark}

\color{black}
Now we are ready to present the following bound for the spectral radius of the sum of products of operators.
\begin{theorem}\label{Thm_1_r}
Let ${{A}_{i}}\in \mathbb B\left( {{\mathbb H}_{i}},{{\mathbb H}_{1}} \right)$ and let ${{B}_{i}}\in \mathbb B\left( {{\mathbb H}_{1}},{{\mathbb H}_{i}} \right)$. Then
\[r\left( \sum\limits_{i=1}^{n}{{{A}_{i}}{{B}_{i}}} \right)\le \omega \left( \left[ {{\gamma }_{ij}} \right]_{i,j=1}^{n} \right),\]
where
{\small
\begin{equation*}
{{\gamma }_{ij}}=\left\{ \begin{array}{ll}
   \omega \left( {{B}_{i}}{{A}_{i}} \right)&\text{ if }i=j \\ 
  \sqrt{\omega \left( \left| {{B}_{i}}{{A}_{j}} \right|+i\left| A_{i}^{*}B_{j}^{*} \right| \right)\omega \left( \left| A_{j}^{*}B_{i}^{*} \right|+i\left| {{B}_{j}}{{A}_{i}} \right| \right)+\left\| {{B}_{i}}{{A}_{j}} \right\|\left\| {{B}_{j}}{{A}_{i}} \right\|+\omega \left( {{B}_{j}}{{A}_{i}}{{B}_{i}}{{A}_{j}} \right)}&\text{ if }i<j \\ 
  0&\text{ if }i>j \\ 
\end{array} \right..
\end{equation*}}
\end{theorem}
\begin{proof}
Letting 
\[A=\left[ \begin{matrix}
   {{A}_{1}} & {{A}_{2}} & \cdots  & {{A}_{n}}  \\
   O & O & \cdots  & O  \\
   \vdots  & \vdots  & \ddots  & \vdots   \\
   O & O & \cdots  & O  \\
\end{matrix} \right],B=\left[ \begin{matrix}
   {{B}_{1}} & O & \cdots  & O  \\
   {{B}_{2}} & O & \cdots  & O  \\
   \vdots  & \vdots  & \ddots  & \vdots   \\
   {{B}_{n}} & O & \cdots  & O  \\
\end{matrix} \right]\in \mathbb B\left( \oplus _{i=1}^{n}{{\mathbb H}_{i}} \right),\]
we get
	\[r\left( \sum\limits_{j=1}^{n}{{{A}_{j}}{{B}_{j}}} \right)=r\left( \sum\limits_{j=1}^{n}{{{A}_{j}}{{B}_{j}}\oplus O} \right)=r\left( AB \right)=r\left( BA \right)\le \omega \left( BA \right).\]
Now, we infer the result by Theorem \ref{5}.
\end{proof}
As an immediate application of Theorem \ref{Thm_1_r}, we derive the following spectral radius bound.
\begin{corollary}
Let $A_1,A_2,B_1,B_2\in\mathbb{B}(\mathbb{H}).$ Then
{\scriptsize	
\[\begin{aligned}
  & r\left( {{A}_{1}}{{B}_{1}}+{{A}_{2}}{{B}_{2}} \right) \\ 
 & \le \frac{1}{2}\left( \omega \left( {{B}_{1}}{{A}_{1}} \right)+\omega \left( {{B}_{2}}{{A}_{2}} \right) \right) \\ 
 &\quad +\frac{1}{2}\sqrt{{{\left( \omega \left( {{B}_{1}}{{A}_{1}} \right)-\omega \left( {{B}_{2}}{{A}_{2}} \right) \right)}^{2}}+\omega \left( \left| {{B}_{1}}{{A}_{2}} \right|+i\left| A_{1}^{*}B_{2}^{*} \right| \right)\omega \left( \left| A_{2}^{*}B_{1}^{*} \right|+i\left| {{B}_{2}}{{A}_{1}} \right| \right)+\left\| {{B}_{1}}{{A}_{2}} \right\|\left\| {{B}_{2}}{{A}_{1}} \right\|+\omega \left( {{B}_{2}}{{A}_{1}}{{B}_{1}}{{A}_{2}} \right)}. 
\end{aligned}\]
}
\end{corollary}

On the other hand, Theorem \ref{6} implies the following other bound.
\begin{theorem}\label{Thm_r_2}
Let ${{A}_{i}}\in \mathbb B\left( {{\mathbb H}_{i}},{{\mathbb H}_{1}} \right)$ and let ${{B}_{i}}\in \mathbb B\left( {{\mathbb H}_{1}},{{\mathbb H}_{i}} \right)$. Then
\[r\left( \sum\limits_{i=1}^{n}{{{A}_{i}}{{B}_{i}}} \right)\le \omega \left( \left[ {{\lambda }_{ij}} \right]_{i,j=1}^{n} \right),\]
where
{\footnotesize
\begin{equation*}
{{\lambda }_{ij}}=\left\{ \begin{array}{ll}
   \omega \left( {{B}_{i}}{{A}_{i}} \right)&\text{ if }i=j \\ 
  \sqrt{\omega \left( \left| {{B}_{i}}{{A}_{j}} \right|+i\left| A_{i}^{*}B_{j}^{*} \right| \right)\omega \left( \left| A_{j}^{*}B_{i}^{*} \right|+i\left| {{B}_{j}}{{A}_{i}} \right| \right)+\frac{1}{2}\left\| {{\left| {{B}_{i}}{{A}_{j}} \right|}^{2}}+{{\left| A_{i}^{*}B_{j}^{*} \right|}^{2}} \right\|+\omega \left( {{B}_{j}}{{A}_{i}}{{B}_{i}}{{A}_{j}} \right)}&\text{ if }i<j \\ 
  0&\text{ if }i>j \\ 
\end{array} \right..
\end{equation*}
}
\end{theorem}
As a consequence of Theorem \ref{Thm_r_2}, we reach the following other bound for the spectral radius.
\begin{corollary}
Let $A_1,A_2,B_1,B_2\in\mathbb{B}(\mathbb{H})$. Then
{\scriptsize
\[\begin{aligned}
  & r\left( {{A}_{1}}{{B}_{1}}+{{A}_{2}}{{B}_{2}} \right) \\ 
 & \le \frac{1}{2}\left( \omega \left( {{B}_{1}}{{A}_{1}} \right)+\omega \left( {{B}_{2}}{{A}_{2}} \right) \right) \\ 
 &\quad +\frac{1}{2}\sqrt{{{\left( \omega \left( {{B}_{1}}{{A}_{1}} \right)-\omega \left( {{B}_{2}}{{A}_{2}} \right) \right)}^{2}}+\omega \left( \left| {{B}_{1}}{{A}_{2}} \right|+i\left| A_{1}^{*}B_{2}^{*} \right| \right)\omega \left( \left| A_{2}^{*}B_{1}^{*} \right|+i\left| {{B}_{2}}{{A}_{1}} \right| \right)+\frac{1}{2}\left\| {{\left| {{B}_{1}}{{A}_{2}} \right|}^{2}}+{{\left| A_{1}^{*}B_{2}^{*} \right|}^{2}} \right\|+\omega \left( {{B}_{2}}{{A}_{1}}{{B}_{1}}{{A}_{2}} \right)}.
\end{aligned}\]
}
\end{corollary}

\section*{Declarations}

\subsection*{Ethical approval}
This declaration is not applicable.
\subsection*{Conflict of interest}
All authors declare that they have no conflicts of interest.
\subsection*{Authors' contribution}
The authors have contributed equally to this work.
\subsection*{Funding}
The authors did not receive any funding to accomplish this work.
\subsection*{Availability of data and materials}
This declaration is not applicable.

%########################

\vskip 0.3 true cm 	

{\tiny (M. Sababheh) Department of Basic Sciences, Princess Sumaya University for Technology, Amman, Jordan
	
\textit{E-mail address:} sababheh@yahoo.com}

\vskip 0.3 true cm 	

{\tiny (H. R. Moradi) Department of Mathematics, Ma.C., Islamic Azad University, Mashhad, Iran 
	
\textit{E-mail address:} hrmoradi@mshdiau.ac.ir}

\end{document}